\documentclass[12pt]{amsart}
\usepackage{amsmath,amsthm,cancel,amssymb,graphicx,amscd,srcltx}

\usepackage[utf8]{inputenc}

\newcommand{\disjunt}{\ltimes}

\newtheorem{proposition}{Proposition}
\newtheorem{lemma}[proposition]{Lemma}
\newtheorem{teorema}[proposition]{Theorem}
\newtheorem{coro}[proposition]{Corollary}

\theoremstyle{definition}

\newtheorem{definition}{Definition}[section]

\theoremstyle{remark}
\newtheorem{remark}{Remark}

\title{Total curvature of complete surfaces in hyperbolic space}
\author{Gil Solanes}
\address{Departament de Matem\`atiques\\Univeristat Autònoma de Barcelona\\Edifici Cc. Campus de la UAB\\08193 Bellaterra}
\email{solanes@mat.uab.cat}
\thanks{Work partially supported by FEDER/MEC grant number MTM2006-04353 and the Ram\'on y Cajal program.}
\begin{document}\maketitle

\begin{abstract}
We prove a Gauss-Bonnet formula for the extrinsic
curvature of complete surfaces in hyperbolic space under some
assumptions on the asymptotic behaviour. The result is given in terms of the measure of geodesics intersecting the surface non-trivially, and of a conformal invariant of the curve at infinity.
\end{abstract}

\section{Introduction and main results}
In this paper we prove a Gauss-Bonnet formula for the total
extrinsic curvature of complete surfaces in hyperbolic space. Our
result is analogous to those obtained by Dillen and K\"uhnel in \cite{kuehnel} for
submanifolds of euclidean space, where the
total curvature of a submanifold $S$ is given in terms of the Euler
characteristic $\chi(S)$, and the geometry of $S$ at infinity (see also Dutertre's work \cite{dutertre} on semi-algebraic sets).

Our starting point is the following well-know equality for $S\looparrowright\mathbb  H^3$, a compact
surface with boundary immersed in hyperbolic 3-space:
\begin{equation}\label{motivacio}
\int_S K dS=2\pi\chi(S)+F(S)-\int_{\partial S}k_g ds
\end{equation}
being $K$ the extrinsic curvature of $S$ (i.e. the product of its principal curvatures), $F(S)$ the area, and $k_g$ the
geodesic curvature of $\partial S$ in $S$. This formula follows from the classical (intrinsic) Gauss-Bonnet theorem, and the Gauss equation.  We plan to make $S$
expand over a complete non-compact surface, but the last two terms in \eqref{motivacio}  are likely to
become infinite. To avoid an indeterminate form, we add and
subtract the \emph{area} enclosed by the curve $\partial S$. Such a notion was defined by
Banchoff and Pohl (cf. \cite{banchoff.pohl} and also \cite{teufel}) for any closed space curve $C$ as
\[
\mathcal A(C):=\frac{1}{\pi}\int_{\mathcal{L}}\lambda^2(\ell,C)d\ell
\]
where $\mathcal{L}$ is (in our case) the space of geodesics in
$\mathbb H^3$, $d\ell$ is the invariant measure on $\mathcal L$ (unique up to normalization), and
$\lambda(\ell,C)$ is the linking number of $C$
with $\ell\in\mathcal L$. This definition was motivated by the
Crofton formula which states
\begin{equation}\label{crofton}
F(S)=\frac{1}{\pi}\int_{\mathcal L}\#(\ell\cap S)d\ell,
\end{equation}
where $\#$ stands for the cardinal.
Hence, we can rewrite \eqref{motivacio} as follows
\begin{equation}\label{type}\notag
\int_S K dS=2\pi\chi(S)+\frac{1}{\pi}\int_{\mathcal{L}}(\#(\ell\cap
S)-\lambda^2(\ell,\partial S))d\ell+\mathcal A(\partial S)-\int_{\partial S}k_gds.
\end{equation}
Our main result is a similar formula  for
complete surfaces in $\mathbb H^3$ defining a smooth curve
$C$ in $\partial_\infty \mathbb H^3$, the ideal boundary of hyperbolic space. In that case, the last two terms of the previous equation are replaced by a conformal (or M\"obius) invariant of the  geometry of $C$ in $\partial_\infty\mathbb H^3$. To be precise, our result applies to
surfaces with {\em cone-like ends} in the sense defined next. A similar
notion of cone-like ends for submanifolds in euclidean space appears in \cite{kuehnel}.

\begin{definition}\label{conelike}
Let $f\colon S\looparrowright \mathbb H^3$ be an immersion of a  $\mathcal C^2$-differentiable surface $S$ in hyperbolic space. We say $S$
has \emph{cone-like ends} if
\begin{enumerate}
\item[i)] $S$ is the interior of a compact surface with boundary  $\overline S$, and taking the
Poincar\'{e} half-space model of hyperbolic space,  $f$ extends to a  $\mathcal C^2$-differentiable immersion $f:\overline{S}\looparrowright\mathbb R^3$,
\item[ii)] $C=f(\partial \overline S)$ is a collection of simple closed curves contained in $\partial_\infty \mathbb H^3$, the boundary of the model, and
\item[iii)] $f(\overline S)$ is orthogonal to $\partial_\infty\mathbb H^3$ along $C$.

\end{enumerate}
\end{definition}

In particular, such a surface is complete with the induced metric. We will
see that surfaces with cone-like ends have finite total extrinsic
curvature. There are also examples of complete non-compact surfaces
with finite total extrinsic curvature which do not fulfill $i)$ or $ii)$ in the previous definition. Condition $iii)$ however is necessary for the total curvature to be finite: the limit of the extrinsic curvature of $S$ at an ideal point $x\in C$  is $\cos^2(\beta)$ where $\beta$ is the angle between $S$ and $\partial_\infty\mathbb H^3$ at $x$.  

In the Klein (or projective) model, the definition reads the same, but
replacing  the word `orthogonal' by `transverse'.  We
will mainly work with the Poincar\'{e} half-space model. Unless
otherwise stated all the metric notions (such as length, area or
curvature) will refer to the hyperbolic metric. 

\bigskip
Given a connected oriented curve $C\subset \partial_\infty\mathbb H^3\equiv
\mathbb R^2$, and a pair of distinct points $x,y\in C$, let us
consider the oriented angle at $x$ from  $C$ to the oriented circle through $x$ that is positively tangent to $C$ at $y$. This angle admits a unique continuous determination
$\theta:C\times C\rightarrow \mathbb R$ that vanishes on the
diagonal. Note that $\theta(y,x)=\theta(x,y)$ and $\theta$ is
independent of the orientation of $C$.

We will prove the following result.
\begin{teorema}\label{main}
Let $S\subset \mathbb H^3$ be a simply connected surface of class $\mathcal C^2$, embedded in the Poincar\'e half-space model of hyperbolic space, and with a (connected) cone-like end $C\subset\partial_\infty\mathbb H^3$.  Then, the integral over $S$ of the extrinsic curvature $K$ is 
\begin{equation}\label{eqmain}
\int_S KdS=\frac{1}{\pi}\int_{\mathcal{L}}(\#(\ell\cap
S)-\lambda^2(\ell,C))d\ell-\frac{1}{\pi}\int_{C\times
C}\theta\sin\theta\frac{dxdy}{\|y-x\|^2}
\end{equation}where \begin{itemize}
\item $d\ell$ is an invariant  measure on the space of geodesics $\mathcal L$,
\item$\lambda^2(\ell,C)$ is 1 if the ideal endpoints of $\ell$ are on different components of
$\partial_\infty \mathbb H^3\setminus C$ and $0$ otherwise, and
\item $dx, dy$ denote length elements on $C$ with respect to the
euclidean metric $\|\cdot \|$ on $\partial_\infty\mathbb H^3\equiv\mathbb
R^2$.
\end{itemize}The
integrals in \eqref{eqmain} are absolutely convergent.
\end{teorema}

\begin{remark}The most interesting term in \eqref{eqmain} is the last one, which we call the {\em ideal defect} of $S$. It defines a functional for plane curves which is invariant under the action of the M\"obius group. In fact, the form $dx dy/\|y-x\|^2$, as well as $\theta(x,y)$, is invariant under M\"obius transformations. 
Similar expressions for space curves appear often in the study of conformally invariant knot energies (cf.\cite{ohara}).

The first term in the right hand side of \eqref{eqmain} is positive, and can be considered as a `truncated area' of $S$, in view of \eqref{crofton}. We call this term the {\em measure of non-trivial geodesics of $S$}. From Proposition \ref{p3}, it will be clear that it is a natural functional of $S$.
\end{remark}

The idea of the proof is roughly the following. We pull-back $d\ell$ to the space of point pairs of $S$. Integration gives the measure of non-trivial geodesics. Applying Stokes' theorem yields then the result. This procedure was already used  by Pohl in the euclidean setting in \cite{pohl}, but here we use a different `primitive' of  $d\ell$. This leads to a somehow dual construction, where the total curvature instead of the area appears. This dual approach is not possible in euclidean space.

\medskip
From Theorem \ref{main} one gets easily a formula for a general
surface with cone-like ends.

\begin{coro}
Let $S\looparrowright \mathbb H^3$ be a $\mathcal C^2$-immersed complete surface with
cone-like ends $C_1,\ldots , C_n$, the curves $C_i$ being
simple and closed. Then
\begin{multline*}
\int_S
KdS=2\pi(\chi(S)-n)+\frac{1}{\pi}\int_{\mathcal{L}}(\#(\ell\cap
S)-\sum_{i=1}^n\lambda^2(\ell,C_i))d\ell\\-\frac{1}{\pi}\sum_{i=1}^n\int_{C_i\times
C_i}\theta\sin\theta\frac{dxdy}{\|y-x\|^2},
\end{multline*}
and the previous integrals are absolutely convergent.
\end{coro}
\begin{proof}
Take a compact set $K\subset\mathbb H^3$ with $\mathcal C^2$ boundary $\partial K$
transverse to $S$, and such that $S\setminus K=S_1\cup \ldots\cup
S_n$, where each $S_i$ is an embedded topological cylinder over $C_i$. Applying \eqref{motivacio} and \eqref{crofton} to 
$R=S\cap K$ yields
\begin{equation}\label{erra}
\int_R KdR=2\pi\chi(R)-\int_{\partial R}
k_g(s)ds+\frac{1}{\pi}\int_{\mathcal L}\#(\ell\cap R)d\ell
\end{equation}
where $k_g$ is the geodesic curvature in $R$.

Let $R_i$ be a compact surface with boundary such that $T_i=R_i\cup
S_i$ is a complete embedded simply connected surface. Combining again \eqref{motivacio} and \eqref{crofton}, gives
\begin{equation}
\label{ri} \int_{R_i} KdR_i=2\pi-\int_{\partial R_i}
k_g(s)ds+\frac{1}{\pi}\int_{\mathcal L}\#(\ell\cap R_i)d\ell.
\end{equation}
Applying Theorem \ref{main} to each $T_i$, and comparing with
\eqref{ri} yields
\begin{eqnarray}
\label{si} \int_{S_i}
KdS_i=-2\pi+\frac{1}{\pi}\int_{\mathcal{L}}(\#(\ell\cap
S_i)-\lambda^2(\ell,C_i))d\ell\\-\frac{1}{\pi}\int_{C_i\times
C_i}\theta\sin\theta\frac{dxdy}{\|y-x\|^2}+\int_{\partial
R_i}k_g(s)ds.\notag
\end{eqnarray}
Addition of \eqref{erra} and \eqref{si} finishes the proof.
\end{proof}

\subsection{The ideal defect}\label{defecte}
The last term in \eqref{eqmain}, which we call the {\em ideal defect}, can also be described as an integral in the space of point pairs of $\partial_\infty \mathbb H^3\equiv \mathbb R^2$, with respect to the M\"obius invariant measure on this space.

\begin{proposition}\label{nt}
 Let $\Omega\subset \mathbb R^2$ be a compact domain bounded by a simple closed curve $C$ of class $\mathcal C^2$. Then
\[
 \int_{C\times C}\theta\sin\theta\frac{dxdy}{\|y-x\|^2}=4\int_{NT(\Omega)}\frac{dzdw}{\|z-w\|^4}
\]
where $NT(\Omega)\subset\Omega\times\Omega$ is the set of point pairs $(z,w)$ such that  any circle $\xi\subset\mathbb R^2$ containing $z$ and $w$ intersects $\mathbb R^2\setminus \Omega$ (i.e. $z,w\in \xi\Rightarrow \xi\not\subset\Omega$.)
\end{proposition}
\begin{proof}
 Let $Q\subset\mathbb H^3$ be the convex hull of $\Omega^c=\partial_\infty\mathbb H^3\setminus \Omega$; i.e. $Q$ is the minimal convex set containing $\Omega^c$. Using the Klein model, $Q$ can be seen as the euclidean convex hull of $\Omega^c$. Let us consider the boundary $S=\partial Q\subset\mathbb H^3$, which is a surface of class $\mathcal C^1$. Next we construct a sequence of convex sets $Q_n\subset\mathbb H^3$ such that: $Q_n\supset Q_{n+1}$, $Q=\cap_{n=1}^\infty Q_n$, and  $S_n=\partial Q_n$ is a  $\mathcal C^2$ surface with cone-like end $C$.  First, let $X\in\mathfrak X(\mathbb R^3)$ be a vector field in the Klein model such that $X$  vanishes only at $C$, and $X|_\Omega$ points to the interior of the model. Then, for small $t>0$, the flow $\varphi_t$ brings $\Omega$ to a surface $\varphi_t(\Omega)$ with a cone-like end on $C$, and bounding a convex domain $D$. On the other hand, let $Q$ be approximated by a decreasing sequence $Q_n'\subset\mathbb R^3$ of euclidean convex sets with boundary of class $\mathcal C^2$ (cf. \cite{schneider}). % Infles $Q$ un radi $\epsilon$ i l'aproximes a distància $\epsilon/4$.
Then, smoothening the corners of $D\cap Q_n'$ yields the desired sequence.

By Theorem \ref{main}
\[
 \int_{S_n}Kd S_n= \frac{1}{\pi}\int_{\mathcal L}(\#(\ell\cap S_n)-\lambda^2(\ell,C))d\ell-\frac{1
}{\pi}\int_{C\times C}\theta\sin\theta\frac{dxdy}{\|y-x\|^2}.
\]
Using, for instance, the arguments in \cite{crasp}, one can show
\[
 \lim_n \int_{S_n}KdS_n=0.
\]
On the other hand, by monotone convergence, 
\[
 \lim_n \int_{\mathcal L}(\#(\ell\cap S_n)-\lambda^2(\ell,C))d\ell=\int_{\mathcal L}(\#(\ell\cap S)-\lambda^2(\ell,C))d\ell.
\]
Hence,
\[
 \int_{C\times C}\theta\sin\theta\frac{dxdy}{\|y-x\|^2}=\int_{\mathcal L}(\#(\ell\cap S)-\lambda^2(\ell,C))d\ell.
\]
The right hand side above is the measure of geodesics intersecting $Q$ but not $\Omega$. We determine each geodesic $\ell\in\mathcal L$ by its ideal endpoints $(z,w)$. This allows to express $d\ell$ as in \eqref{dzdw}.
Finally, we just need to note that a geodesic $\ell$ intersects the convex hull $Q$ if and only if every geodesic 2-plane containing $\ell$ intersects  $\Omega$. 
\end{proof}

\subsection{Integral of the inverse of the chord}\label{inversa}
Next we express the ideal defect in an alternative way which is not invariant, but still interesting.  Let $C\subset \partial_\infty\mathbb H^3$ be a $\mathcal C^2$-differentiable simple closed 
curve, and consider
$S=C\times(0,\infty)\subset\mathbb H^3$. We may think of $S$ as a
surface with one end by closing the top end at infinity with an
infinitesimally small surface. Then, the total curvature of $S$ equals $2\pi$, and Theorem \ref{main} applied to
$S$ yields
\[
2\pi+\frac{1}{\pi}\int_{C\times C}\theta\sin\theta
\frac{dxdy}{\|y-x\|^2}=\frac{2}{\pi}\int_{\mathbb R^2\times\mathbb
R^2}(\#(\overline{zw}\cap C)-\lambda^2(z,w;C))\frac{dzdw}{\|w-z\|^4}
\]\begin{equation}\label{exemple}
=\frac{2}{\pi}\int_{A(2,1)}\sum_{x, y\in L\cap
C}\frac{(-1)^{\#(\overline{xy}\cap C)}}{\|y-x\|}dL
\end{equation}
where $\overline{zw}$ denotes the line segment joining $z,w\in \mathbb R^2$, and $dL$ is the invariant measure on the space $A(2,1)$ of
(unoriented) lines of $\mathbb R^2$, normalized as in \cite{santalo}.
The first equality uses \eqref{dzdw}. The second equality is Proposition \ref{lemaconv}. 

As a consequence, the integral in \eqref{exemple} is invariant under M\"{o}bius transformations, which was
\emph{a priori} not obvious. In fact, if $C$ bounds a convex domain
$\Omega$, then \eqref{exemple} is
\begin{equation}\frac{4}{\pi}\int_{A(2,1)}\frac{1}{\sigma(L\cap\Omega)}dL
\label{corda}
\end{equation}
where $\sigma(L\cap\Omega)$ is the chord length. The previous
functional \eqref{corda} is one of the so-called Franklin invariants
of convex sets, defined by Santal\'{o} in \cite{santalo.franklin} as a
generalization of a functional introduced by Franklin with motivations from
stereology (cf.~\cite{franklin}). These functionals had the nice
property of being invariant by dilatations. For instance, the integral \eqref{corda} could in principle be used to estimate, by means of line sections, the number of particles in a plane region, if these particles have the same shape but possibly different size.

An immediate consequence
of our results is that \eqref{corda} is in fact invariant under the 
M\"{o}bius group. An interesting question is to determine which of the
Franklin functionals enjoy this bigger invariance. Besides, it was
conjectured that the Franklin invariants are minimal for balls
(cf.~\cite{franklin} and \cite{santalo.franklin}). This was shown by
Franklin among ellipsoids while Santal\'{o} obtained some general non-sharp
inequalities. As a consequence of our results, we can prove this
conjecture in the planar case.
\begin{coro}
For a convex set $\Omega\subset\mathbb R^2$ we have
\begin{equation}
\label{conjectura}\int_{A(2,1)}\frac{1}{\sigma(L\cap \Omega)}dL\geq
\frac{\pi^2}{2}
\end{equation}
where $\sigma$ is the length of the chord, and $A(2,1)$ is the space
of lines. Equality holds in \eqref{conjectura} if
and only if $\Omega$ is a round disk. Moreover, the left hand side
of \eqref{conjectura} is invariant by M\"{o}bius transformations (keeping $\Omega$ convex).
\end{coro}
\begin{proof}By \eqref{exemple} we have
\[
\frac{4}{\pi}\int_{A(2,1)}\frac{1}{\sigma(L\cap
\Omega)}dL=2\pi+\frac{1}{\pi}\int_{C\times
C}\theta\sin\theta\frac{dxdy}{\|y-x\|^2}\geq 2\pi,
\]
and the equality occurs if and only if $\theta\equiv 0$. Indeed,
since $C$ is convex it is easy to see that $-\pi<\theta<\pi$.
\end{proof}

The author wishes to thank W.K\"{u}hnel, R.Langevin, J.O'Hara and
E.Teufel for valuable discussions during the preparation of this
work. It is a pleasure to thank the anonymous referee for useful suggestions that have led to a substantial simplification of the paper.

\section{The space of geodesics}\label{geodesics}
Let $\mathcal F=\{(x;g_1,g_2, g_3)\}$ be the bundle of positive
orthonormal frames of $\mathbb H^3$; i.e., each $(g_i)_{i=1,2,3}$ is
a positive orthonormal basis of $T_x\mathbb H^3$.  We consider on
$\mathcal F$ the dual and  connection forms
\[
\omega_i=\langle dx,g_i\rangle, \qquad \omega_{ij}=\langle \nabla
g_i,g_j\rangle,
\]
where $\langle\,,\,\rangle$ denotes the (hyperbolic) metric in
$\mathbb H^3$, and $\nabla$ is the corresponding riemannian
connection. The structure equations read
\begin{equation}
\label{estructura}d\omega_i=\omega_j\wedge\omega_{ji},\qquad
d\omega_{ij}=\omega_i\wedge\omega_j+\omega_{ik}\wedge\omega_{kj}.
\end{equation}

Let $\mathcal L^+$ be the space of oriented geodesics of $\mathbb
H^3$. Clearly $\mathcal L^+$ is a double cover of $\mathcal L$. Consider $\pi_1:\mathcal F\rightarrow \mathcal L^+$
given by $\pi_1(x;g_1,g_2,g_3)=\ell$ with $x\in \ell$,
and $g_1\in T_x\ell$ pointing in the positive direction. The space $\mathcal L^+$ can be endowed with a differentiable structure such that
$\pi_1$ is a smooth submersion. Moreover, $\mathcal L^+$
admits a volume form $d\ell$ invariant under isometries of $\mathbb
H^3$, which is unique up to normalization, and characterized by (cf.\cite{santalo})
\begin{equation}\label{mesura}
\pi_1^*(d\ell)=\omega_{2}\wedge\omega_{12}\wedge\omega_3\wedge\omega_{13}.
\end{equation}
Similarly, one can consider $\mathcal L_2$, the space of (unoriented)
totally geodesic surfaces (geodesic planes) of $\mathbb H^3$. We
will use the space  of {\em flags}
\[
\mathcal L_{1,2}=\{(\ell,\wp)\in\mathcal
L^+\times\mathcal L_2|\ell\subset\wp\},
\]
and the canonical
projection $\pi\colon\mathcal L_{1,2}\rightarrow \mathcal L^+$
which makes $\mathcal L_{1,2}$ a principal $\mathbb S^1$-bundle over
$\mathcal L^+$. Let us project $\pi_{1,2}\colon\mathcal F\rightarrow \mathcal
L_{1,2}$ so that $\pi_{1,2}(x;g_i)=(\ell,\wp)$ with
$\wp\supset\ell=\pi_1(x;g_i)$ and $g_3\bot T_x\wp$. Then $\omega_{23}=\pi_{1,2}^*\varphi$ for a certain form
$\varphi\in\Omega^1(\mathcal L_{1,2})$, which is an invariant global angular form
 (or connection) of the bundle $\pi$.

\begin{proposition}\label{p1}
There exists a unique $2$-form $\alpha\in\Omega^2(\mathcal{L}^+)$ such
that
\[
\pi^*(\alpha)=d\varphi\in\Omega^2(\mathcal L_{1,2}).
\] where $\varphi$ is the global angular form of $\pi$. Moreover $\alpha\wedge\alpha=2d\ell$, so that $\alpha$ is an invariant symplectic form on $\mathcal L^+$.\end{proposition}
\begin{proof}Assuming $\alpha$ exists, structure equations \eqref{estructura} give
\begin{equation}\label{alpha}
\pi_1^*(\alpha)=d\omega_{23}=\omega_2\wedge\omega_3-\omega_{12}\wedge\omega_{13},
\end{equation}
whence
\[
\pi_1^*(\alpha\wedge\alpha)=-2\omega_2\wedge\omega_3\wedge\omega_{12}\wedge\omega_{13}=2\pi_1^*(d\ell).
\]
Therefore  $\alpha\wedge\alpha=2d\ell$ (as $d\pi_1$ is exhaustive).

Let $X\in\mathfrak X(\mathcal L_{1,2})$ be the tangent vector field along
the fibers of $\pi$ such that $\varphi(X)=1$. By \eqref{alpha}, for any $\widetilde X\in\mathfrak X(\mathcal F)$ such that $d\pi_{1,2}\widetilde X=X$,
\[
 \pi_{1,2}^*(i_Xd\varphi)=i_{\widetilde X}d\omega_{23}=0,
\]
whence $i_Xd\varphi=0$. Then $L_X\varphi=0$, and
\[
L_Xd\varphi=d L_X\varphi=0.
\]Hence, $d\varphi$ is constant along the fibers of $\pi$, and null
on their tangent vectors, which shows the existence of $\alpha$. The uniqueness follows from the injectivity of $\pi^*$.
\end{proof}

It follows from the previous proposition that 
\begin{equation}\label{primitiva}
d( \pi^*\alpha\wedge
\varphi)=2\cdot \pi^*(d\ell)
\end{equation}
This will be used in Section \ref{proof} to prove Theorem \ref{main} by means of Stokes' theorem.

\begin{remark}The forms $\varphi,\alpha$ are in some sense dual to the forms $\omega_1,dI$ used in \cite{pohl}. In fact, many of the subsequent constructions are parallel to those of \cite{pohl}. However, choosing $\alpha$ leads us to results involving the total curvature, while $dI$ made the area appear. This choice could not be done in the euclidean setting since there $\alpha\wedge\alpha$ vanishes.\end{remark}

The following notation will be used throughout the paper: $$A\disjunt B:=\{(x,y)\in A\times B\ |\ x\neq y\}.$$
In the Poincaré model, by considering the ideal endpoints $z,w$ of each geodesic $\ell$, one identifies (a full-measure subset of) $\mathcal{L}^+$ with $\mathbb {R}^2\disjunt\mathbb R^2$. Then, an elementary computation with moving frames (cf.\eqref{connexio}) gives the following expression for $\alpha$ at a point $(z,w)\equiv\ell\in\mathcal L^+$:
\begin{equation}\label{alpha1}
\alpha=\frac{2}{\|w-z\|^2} ( dz_1\wedge dw_2+dz_2\wedge dw_1)
\end{equation}
where the coordinate system of $\mathbb R^2$ has been chosen in such a way that $z_2=w_2=0$ and $w_1=-z_1$. In particular, if $z=z(x)$ and $w=w(y)$ are curves parametrized by arc-length, then
\begin{equation}\label{alpha2}
 \alpha=2\sin\theta(x,y)\  \frac{dx\wedge dy}{\|w(y)-z(x)\|^2}
\end{equation}
where $\theta(x,y)$ is the oriented angle between the two oriented circles through $z(x),w(y)$, tangent to $z'(x)$ and $w'(y)$ respectively.

Using \eqref{alpha1} we can also obtain an expression for the measure of geodesics. Indeed,
\begin{equation}\label{dzdw}
 d\ell=\frac{1}{2}\alpha\wedge\alpha=4\frac{dz\wedge dw}{\|w-z\|^4}
\end{equation}
where $dz,dw$ denote the area elements of the ideal endpoints $z,w$ in $\mathbb R^2$.
\begin{remark}The following complex valued two form in $\mathbb C\disjunt\mathbb
C$ was introduced by Langevin and O'Hara in
\cite{ohara} under the name  \emph{infinitesimal cross-ratio}
\[
\omega_{cr}=\frac{d(z_1+iz_2)\wedge d(w_1+iw_2)}{(w-z)^2},\qquad (z,w)\in\mathbb C\disjunt\mathbb C.
\]
This form $\omega_{cr}$ is invariant under the diagonal action of the Möbius group $Sl(2,\mathbb C)$. Using this fact, one checks easily  that $-\alpha/2$ coincides with $\mathfrak{Im}(\omega_{cr})$, the imaginary part of the infinitesimal cross-ratio.
\end{remark}

\medskip
We end the section by showing that  the measure of non-trivial geodesics is a natural quantity. This fact was already noticed in the euclidean setting by Pohl (cf. \cite{pohl}, equation (6.5)).
\begin{proposition}\label{p3}
Let $S\subset\mathbb H^3$ be an embedded surface with cone-like ends $C\subset\partial_\infty\mathbb H^3$. Let $\Phi: S\disjunt S\rightarrow \mathcal L^+$ be such that $\Phi(x,y)$ is the oriented geodesic going first
through $x$ and then through $y$. Then
\[
\int_{S\disjunt S}\Phi^*(d\ell)=\frac{1}{2}\int_{\mathcal{L}^+}(\#(\ell\cap
S)-\lambda^2(\ell,C))d\ell.
\]
\end{proposition}
\begin{proof}By the coarea
formula
\[
\int_{S\disjunt S}\Phi^*(d\ell)=\int_{\mathcal L^+}\mu(\ell)d\ell
\]
where
\[
\mu(\ell)=\sum_{(x,y)\in\Phi^{-1}(\ell)}-\epsilon(x)\epsilon(y)
\]
being $\epsilon(u)$ the
sign at $u$ of the algebraic intersection $\ell\cdot S$. 
Now, let $p$ (resp. $q$) be the number of points of $\ell\cap S$
with $\epsilon=1$ (resp. $\epsilon=-1$), so that
\[
\#(\ell\cap S)=p+q,\qquad \lambda(\ell,C)=\ell\cdot S=p-q.
\]
Then $\Phi^{-1}(\ell)$ contains $(p(p-1)+q(q-1))/2$ pairs $(x,y)$ with $\epsilon(x)=\epsilon(y)$, and $pq$ elements with $\epsilon(x)=-\epsilon(y)$.
Therefore
$2\mu(\ell)=2pq-p(p-1)-q(q-1)=\#(\ell\cap S)-\lambda^2(\ell,C)$.
\end{proof}

\section{Convergence results}\label{convergencia}
Next we establish the convergence of the integrals appearing in Theorem \ref{main}. In the whole section, $S\subset \mathbb H^3$ will denote  a complete surface with a connected cone-like end $C\subset\partial_\infty \mathbb H^3$. Here $\mathbb H^3$ denotes the Poincar\'e half-space model. For $h>0$, we set
$S_h=\{x\in S|x_3\geq h\}$ which is a compact surface with 
boundary $C_h=\partial S_h=\{x\in S|x_3=h\}$.

\begin{proposition}\label{p7}If $K$ denotes the extrinsic curvature of
$S$, and $dS$ is the area element, then
\[
\int_SKdS
\]
is absolutely convergent.
\end{proposition}
\begin{proof}Let us consider the global orthonormal frame $e_i(x)=x_3\partial/\partial x_i$, $(i=1,2,3)$ defined for all $x\in\mathbb H^3$. The
connection forms $\theta_{ij}=\langle \nabla e_i,e_j\rangle$ are then
given by
\begin{equation}\label{connexio}
\theta_{i3}=\frac{dx_i}{x_3},\quad \theta_{ij}=0\qquad \mbox{for } i,j\neq 3.
\end{equation}
Let us fix now $y\in S$. After a change of coordinates, we can
assume $e_2(y)\in T_yS$. 
Let $v_1,v_2,v_3$ be a frame locally defined on $S$ (around $y$) so
that $v_2(y)=e_2(y)$, and $v_1(x),v_2(x)\in T_xS$. Then $v_i(x)=a_{ij}(x)e_j(x)$ for an
orthogonal matrix $(a_{ij}(x))\in \mathrm O(3)$. In particular
$v_1(y)=\cos\alpha e_1+\sin\alpha e_3$, and $v_3(y)=-\sin\alpha
e_1+\cos\alpha e_3$ for some $\alpha\in[0,2\pi)$. Then
$(\omega_{ij})_y=\langle \nabla v_i,v_j\rangle_y$ are given by
\begin{align}\notag
 (\omega_{12})_y&=\langle \nabla (a_{1i}e_i),e_2\rangle_y=da_{12}+a_{12}(y)\langle \nabla e_1,e_2\rangle +a_{32}(y)\langle \nabla e_3,e_2\rangle\\&=da_{12}+\cos\alpha\theta_{12}+\sin\alpha\theta_{32}\stackrel{\eqref{connexio}}{=}da_{12}-\sin\alpha\frac{dx_2}{y_3},\label{udos}
\end{align}
\begin{align*}
(\omega_{13})_y&=\langle\nabla (a_{1i}e_i),-\sin\alpha e_1+\cos\alpha e_3\rangle_y\\
&=-\sin\alpha(da_{11}+a_{1i}(y)\theta_{i1})+\cos\alpha(da_{13}+a_{1i}(y)\theta_{i3})\\
&=-\sin\alpha da_{11}+\cos\alpha da_{13}+\frac{dx_1}{y_3},%
\end{align*}and similarly
\begin{equation}\label{dostres}
(\omega_{23})_y=-\sin\alpha da_{21}+\cos\alpha
da_{23}+{\cos\alpha}\frac{dx_2}{y_3}.
\end{equation}
In particular 
\[
 \omega_{13}(v_1)=-\sin\alpha da_{11}(v_1)+\cos\alpha da_{13}(v_1)+\frac{dx_1(v_1)}{y_3}=O(y_3)+\frac{dx_1(v_1)}{y_3}.
\]
But $dx_1(v_1/y_3)=\cos\alpha=O(y_3)$. Indeed, $\cos\alpha=\langle e_1,v_1\rangle$ is a $\mathcal C^1$ function on $\overline S=S\cup C$ and vanishes at $C$.
One checks similarly that $\omega_{i3}(v_j)=O(y_3)$ for $i,j=1,2$.  We have thus that
\[
 K(y)=\det(\omega_{i3}(v_j)|i,j=1,2)=\omega_{13}(v_1)\omega_{23}(v_2)-\omega_{13}(v_2)\omega_{23}(v_1)=O(y_3^2)
\]
The result follows since  $y_3^2dS$  is the euclidean area element of $S$ in the model.
%.
\end{proof}
The following proposition is a first step towards the existence of formula \eqref{eqmain}.

\begin{proposition}\label{primer}
 Let $S, R\subset \mathbb H^3$ be two surfaces with the same cone-like end $\partial_\infty S=\partial_\infty R\subset \partial_\infty\mathbb H^3$. Then
\[
\int_S K dS-\int_R K dR=2\pi(\chi(S)-\chi(R))+\lim_{h\to 0}\frac{1}{\pi}\int_{\mathcal L}(\#(\ell\cap S_h)-\#(\ell\cap R_h))d\ell. 
\]
\end{proposition}
%  {\tt cal que coincideixin les orientacions indu\"\i des a $C$?}
\begin{proof}
 From \eqref{motivacio} and \eqref{crofton} one gets
\[
 \int_{S_h} K dS_h=2\pi\chi(S_h)+\frac{1}{\pi}\int_{\mathcal L}\#(\ell\cap S_h)d\ell-\int_{\partial S_h}k_g(s)ds
\]
and similarly for $R_h$. We must show that 
\[
 \int_{\partial S_h}k_g(s)ds-\int_{\partial R_h}k_g(s) ds
\]
tends to zero as $h\to 0$. By equation \eqref{udos} we have 
\[
k_g=-\omega_{12}(v_2)=-da_{12}(v_2)+\sin\alpha.\] 
In the previous proof we learned that $\cos\alpha=O(h)$, and thus $\sin\alpha=1+O(h^2)$. Besides, in the choice of the local frame $v_1,v_2,v_3$ one could further assume that $v_1$ is everywhere orthogonal to $e_2$. Hence $a_{12}=\langle v_1,e_2\rangle\equiv 0$, so $k_g=\sin\alpha=1+O(h^2)$, and
\[
 \int_{\partial S_h}(k_g(s)-1)ds=\int_{\partial S_h}O(h^2)ds=O(h),
\]
and similarly for $\partial R_h$. Thus, it suffices to show that the difference of (hyperbolic) lengths of $\partial S_h$ and $\partial R_h$ tends to zero as $h\to 0$. This follows from the fact that $\partial_\infty S$ is an euclidean geodesic of both $S$ and $R$, and geodesics are extremals of the length. Indeed, the euclidean lengths of $\partial S_h$ and $\partial R_h$ differ both from the length of $\partial_\infty S$ with an order $O(h^2)$. Hence, their respective hyperbolic lengths have a difference of order $O(h)$. 
\end{proof}

Next we study the  convergence  of the measure of non-trivial geodesics.

\begin{lemma}\label{p8}If $\lambda(\ell,C_h)$ denotes the linking number (defined up to sign)
of a geodesic $\ell$ with the  curve $C_h$,  then
\begin{equation}\label{mono}
\lim_{h\rightarrow 0}\int_{\mathcal{L}}(\#(\ell\cap
S_h)-\lambda^2(\ell,C_h))d\ell=\int_{\mathcal{L}}(\#(\ell\cap
S)-\lambda^2(\ell,C))d\ell
\end{equation}
where $\lambda(\ell,C)$ is the limit of $\lambda(\ell,C_h)$ when
$h\rightarrow 0$.
\end{lemma}

\begin{proof}
Let $\ell\in\mathcal L$ be transverse to $S$, which happens for almost every $\ell$. Then $\#(\ell\cap S_h)$ is an increasing function of $h$. For $h$ small enough, $C_h$ is connected, and thus $\lambda^2(\ell,C_h)\leq 1$. Therefore
$\#(\ell\cap S_h)-\lambda^2(\ell, C_h)$ is an increasing function of
$h$. Then \eqref{mono} follows by monotone convergence.
\end{proof}
We will see below, that the limit in \eqref{mono} is finite. For the moment, we show this fact for the  infinite cylinder over $C$.

\begin{proposition}\label{lemaconv}Let $C\subset\partial_\infty\mathbb H^3$ be a simple closed curve, and let $R={C\times (0,\infty)} \subset\mathbb H^3$. Then the following integrals converge and coincide
\[ \int_{\mathcal L}(\#(\ell \cap R)-\lambda^2(\ell,C))d\ell=\int_{A(2,1)}\sum_{x, y\in L\cap
C}\frac{(-1)^{\#(\overline{xy}\cap C)}}{\|y-x\|}dL<\infty\]
where $dL$ is an invariant measure in the space $A(2,1)$ of lines in $\mathbb R^2$. 
\end{proposition}
\begin{proof}After a vertical projection onto $\partial_\infty\mathbb H^3$, each geodesic $\ell$ is mapped to a segment $\overline{zw}$, and $R$ projects onto $C$. From the proof of Proposition \ref{p3} we know
\[
 \#(\ell\cap R)-\lambda^2(\ell,C)=-\sum_{x,y\in\overline{zw}\cap C}\epsilon(x)\epsilon(y)
\]
where $\epsilon(u)$ is the
sign at $u$ of the algebraic intersection $\overline{zw}\cdot C$. The equality of the integrals follows from \eqref{dzdw}, together with  (cf.\cite{santalo}, equation (4.2)) 
\[dzdw=\|t-s\|dsdtdL\] where $s,t$ are arc-length parameters of $z,w$
along $L$. 
In order to check the convergence, we use the following expression of the measure of lines in $\mathbb R^2$ (cf. \cite{pohl})
\[
 dL=|\sin\beta_x\sin\beta_y| \frac{dxdy}{\|y-x\|}
\]
where $x,y$ are intersection points with $C$, and $\beta_x,\beta_y$ are the oriented angles between $L$ and $C$ at $x,y$ respectively. Then the integral over $A(2,1)$ above becomes
\[
 -\int_{C\times C}\sin\beta_x\sin\beta_y \frac{dxdy}{\|y-x\|^2}.
\]
This integral converges since $\beta_x,\beta_y=O(\|y-x\|)$ as one can easily prove.
\end{proof}

\begin{lemma}\label{lemaunif}  Let $S, R\subset \mathbb H^3$ be two surfaces with the same cone-like end $\partial_\infty S=\partial_\infty R\subset \partial_\infty\mathbb H^3$. Then the following integrals are uniformly bounded for all $h>0$
\[
 \int_{\mathcal L}(\lambda^2(\ell,\partial S_h)-\lambda^2(\ell,\partial R_h))d\ell.
\]
\end{lemma}
\begin{proof}
Let $T_h$ be the region of $\{x\in\mathbb H^3|x_3=h\}$ bounded by $\partial S_h$ and $\partial R_h$. If a  geodesic $\ell$ is disjoint from $T_h$, then $\lambda^2(\ell,\partial S_h)=\lambda^2(\ell,\partial T_h)$. Hence, the integral above is bounded by the measure of geodesics intersecting $T_h$. By the Crofton formula \eqref{crofton}, this measure is proportional to the area of $T_h$. Since $S$ and $R$ are tangent at infinity, the euclidean area of $T_h$ has order $O(h^2)$. Therefore, its hyperbolic area is uniformly bounded.
\end{proof}

\begin{proposition}\label{convabs}The measure of non-trivial geodesics
\[
 \int_{\mathcal L}(\#(\ell\cap S)-\lambda^2(\ell,C))d\ell 
\]
is absolutely convergent.
\end{proposition}
\begin{proof}Clearly
 \begin{multline*}
  \int_{\mathcal L}(\#(\ell\cap S_h)-\lambda^2(\ell,\partial S_h))d\ell=\int_{\mathcal L}(\#(\ell\cap S_h)-\#(\ell\cap R_h))d\ell\\ + \int_{\mathcal L}(\#(\ell \cap R_h)-\lambda^2(\ell,\partial R_h))d\ell +\int_{\mathcal L}(\lambda^2(\ell,\partial R_h)-\lambda^2(\ell,\partial S_h))d\ell.
 \end{multline*}
The last three integrals are uniformly bounded by Propositions \ref{primer}, and \ref{lemaconv} and  Lemma \ref{lemaunif} respectively. Thus, by monotonicity, the following limit
 \[
  \lim_{h\to 0} \int_{\mathcal L}(\#(\ell\cap S_h)-\lambda^2(\ell,\partial S_h))d\ell,
 \]
exists and is finite. Since $\#(\ell\cap S)-\lambda^2(\ell,C)$ is positive, Lemma \ref{p8} shows the absolute convergence of the integral.
\end{proof}

\begin{coro}\label{apriori}
Let $S\subset\mathbb H^3$ be a surface with a cone-like end $C\subset\partial_\infty\mathbb H^3$. Then
 \[
  \int_S KdS=2\pi\chi(S)+\frac{1}{\pi}\int_{\mathcal L}(\#(\ell\cap S)-\lambda^2(\ell,C))d\ell-\delta(C)
 \]
where $\delta(C)$ depends only on the ideal curve $C$. All the integrals above are absolutely convergent.
\end{coro}

\begin{proof}
 The convergence has been established in Propositions \ref{p7} and \ref{convabs}. The result follows then from Proposition \ref{primer}.
\end{proof}

\begin{remark}
 We have assumed $C$ to be connected for simplicity. If $C$ is a collection of \emph{disjoint} simple closed curves, each of them arbitrarily oriented, the previous results hold without change. The key fact for the convergence is that $\lambda^2(\cdot,C)\leq 1$ outside a compact subset of $\mathcal{L}$. As for $\delta(C)$, it depends in this case on the orientations of $C$, as well as the relative positions of the several components.
\end{remark}

\begin{remark}
 In order to get explicit expressions of $\delta(C)$, it is enough to find, for each curve $C\subset \partial_\infty \mathbb H^3$, a surface $S$ with cone-like ends on $C$ for which the total curvature and the measure of non-trivial geodesics can be computed. In fact, this is what we did in subsections \ref{defecte} and \ref{inversa}.

\smallskip
However, in order to get the expression of $\delta(C)$ that appears in Theorem \ref{main}, we will need to follow a different strategy. 
\end{remark}

\section{Proof of Theorem \ref{main}}\label{proof}
\subsection{The space of chords}
Given a $\mathcal C^2$-differentiable manifold $S$ (without boundary), the \textit{space of
chords} of $S$ is a $\mathcal C^1$-differentiable manifold $M_S$ with boundary,  introduced by Whitney in \cite{whitney}, and described in detail in \cite{pohl}. This space is the blow-up of $S\times S$ along the diagonal. In particular, the interior of $M_S$ is $S\disjunt S$, and the boundary is the sphere bundle of oriented tangent directions of $S$
\[
\partial M_S=T^+S:=(TS\setminus \{(x,\vec 0)|x\in S\})/\mathbb R^+.
\]The reader is referred to \cite{pohl} for details on the differentiable structure of $M_S$. The following property describes this structure quite well:
given a regular injective $\mathcal C^2$-differentiable curve $x\colon [0,1)\rightarrow S$, the curve $c\colon(0,1)\rightarrow S\disjunt S$ defined by $c(t)=(x(0),x(t))$  extends to a $\mathcal C^1$-differentiable curve $c\colon [0,1)\rightarrow M_S$ which meets $\partial M_{S}$ transversely at $c(0)=[x'(0)]\in T^+S$. Another basic property is the following: the natural projections $p_1, p_2\colon S\disjunt S\rightarrow S$ extend naturally to differentiable submersions $p_1,p_2\colon  M_S\rightarrow S$. 

Let now $ S$ be a manifold with boundary. The space  $M_{ S}$ of chords of $ S$ is constructed as follows. We consider a manifold without boundary $\tilde S$ extending $ S$. Let $p_1,p_2:M_{\tilde S}\rightarrow \tilde S$ be the submersions mentioned above. The space of chords of $ S$ is then defined as $M_{ S}=p_1^{-1}( S)\cap p_2^{-1}( S)\subset M_{\tilde S}$ (i.e. $M_{ S}$ contains the chords of $\tilde S$ with both ends in $ S$). This space is a topological manifold with boundary, but this boundary is not smooth. Indeed, the interior of $M_S$ is $S\disjunt S$, and the boundary is $\partial M_{ S}=T^+ S\cup ( S\disjunt \partial S)\cup (\partial  S\disjunt S)$. The faces $T^+S$, $S\disjunt \partial S$, and $\partial  S\disjunt S$ are pairwise transverse outside $T^+\partial S$ (in fact, $\partial  S\disjunt S$ and $S\disjunt \partial S$ are tangent at points of $T^+\partial S$). Hence, $M_S\setminus T^+\partial S$ is a manifold with corners in the usual sense (cf. for instance \cite{lee}).

\subsection{Bundles and sections}
In this subsection we use the Klein model of hyperbolic space. Hence $\mathbb H^3$ is the interior of the closed unit ball $\mathbb B^3$ in $\mathbb R^3$. 
Let $\Psi:\mathbb B^3\disjunt\mathbb B^3\rightarrow \mathcal L^+$ be such that $(x,y)$ is mapped to the geodesic line going first through $x$ and then through $y$. This map extends naturally to $\Psi:M_{\mathbb B^3}\setminus T^+\mathbb S^2\rightarrow \mathcal L^+$. This extension is smooth by the results of  \cite{pohl}.

Let now $S^\circ\subset \mathbb H^3$ be a simply connected surface with a cone-like end $C\subset \partial_\infty\mathbb H^3$.  Then, the closure $S=S^\circ \cup C$ is a compact surface with boundary in  $\mathbb B^3$, transverse to the ideal sphere $\mathbb S^2=\partial_\infty \mathbb H^3$. Notice that we slightly  modified,  for simplicity, the notation used in the previous sections.

As seen in \cite{pohl}, the inclusion $M_S\subset M_{\mathbb B^3}$ is compatible with the differentiable structures. Hence, the mapping
\[\Phi\colon M_{S}\setminus T^+C\rightarrow\mathcal L^+\] obtained as a restriction of $\Psi$ is smooth. Note that this extends the mapping $\Phi$ defined in Proposition \ref{p3}.

\medskip
To simplify the notation  we denote  $B:=M_S\setminus T^+C$. By Proposition \ref{p3}, the measure of non-trival geodesics can be obtained by integrating $\Phi^*(d\ell)$ on $B$. Our aim is to compute this integral by means of Stokes' theorem, using an \emph{invariant} form whose differential is $d\ell$. Such a form is given by \eqref{primitiva}, but it lives in the bundle $\mathcal L_{1,2}$. In fact, there is no invariant form in $\mathcal L^+$ whose differential is $d\ell$. We are thus led to consider the
pull-back by $\Phi$ of the $\mathbb S^1$-bundle $\pi\colon
\mathcal{L}_{1,2}\rightarrow \mathcal L^+$. More precisely, we consider $E:=\Phi^*(\mathcal L_{1,2})=\{(z,\wp)\in B\times \mathcal
L_{2}\ |\ \Phi(z)\subset \wp\}$, and the following commutative diagram with the obvious mappings
\begin{equation} \label{diagrama}\begin{CD}
E @>\Phi' >> \mathcal L_{1,2}\\
@V{\Phi^*\pi}VV @VV{\pi}V\\
{B} @>\Phi>> {\mathcal L^+}
\end{CD}
\end{equation}

\medskip
It would be desirable to define a section of $\Phi^*\pi\colon E\rightarrow B$. This section should be canonically constructed in some geometric way. This can be done quite naturally, but only at the boundary $\partial B$; in fact only on
\[
 \partial B\setminus (C\disjunt C)=(T^+S\setminus T^+C)\cup (S^\circ\times C)\cup (C\times S^\circ)=\partial M_S\setminus M_C.
\]
Indeed, for $z=(x,[v])\in T^+S\setminus T^+C$ we choose the geodesic plane $\wp(z)$ spanned by $T_{x}S$. For $z=(x,y)\in C\times S^\circ$, and for $z=(y,x)\in S^\circ\times C$,  we choose the plane $\wp(z)$ tangent to $C$ at $x$ and containing $y$. Note that this definition does not extend to $C\disjunt C$: the two planes through $x,y\in C$ that are tangent to $C$ at $x$ and $y$ respectively, form a certain angle. In fact, this is precisely the angle $\theta$ appearing in Theorem \ref{main}.

To summarize, we have defined
\begin{eqnarray}\label{seccio} s\colon\partial M_S\setminus M_C&\longrightarrow&
E\\ z&\longmapsto&(z,\wp(z)))\notag
\end{eqnarray}
 in such a way that $T_{x}\wp(z)= T_{x}S$ if $z=(x,[v])\in T^+S$, and $T_{x}C\subset T_{x}\wp(z)$ for $z=(x,y)\in S^\circ\times C$, or $z=(y,x)\in C\times S^\circ$. 

We already noted that $s$ has a jump discontinuity in $C\disjunt C$. To solve this, we shall complete the image of $s$ with a family of fiber intervals interpolating the two one-sided limits  of $s$. However, these intervals are not well-defined in the $\mathbb S^1$-bundle $E$. We are led to consider an infinite cyclic cover of $E$ that gives an $\mathbb R$-bundle over $B$. Next we define this cover, and we show it admits a lift of $s$. Here we  take great advantage of the assumption that $S$ is simply connected.

\begin{proposition}
\label{P4}  The principal $\ \mathbb S^1$-bundle $\Phi^*\pi\colon E\longrightarrow B$ is trivial. Moreover, there is a bundle isomorphism
$\tau\colon E\longrightarrow B\times \mathbb S^1$, such that $\tau\circ s$ lifts over the covering $q\colon B\times \mathbb R\rightarrow B\times \mathbb S^1$; i.e., there exists  a continuous function
\[g\colon\partial M_S\setminus M_C\rightarrow \mathbb R
\] such that $q(x,g(x))=\tau\circ s(x)$ for every $x\in \partial M_S\setminus M_C$.\end{proposition}

\begin{proof}Consider an isotopy of embeddings $H:S\times [0,1]\rightarrow \mathbb B^3$ such that $H_0=id$ and $H_1(S^\circ)$ is contained in a plane $\wp\in\mathcal L_2$. We may construct the isotopy so that $H(C\times [0,1])\subset\mathbb S^2$. Put $\tilde H(x,y,t):=(H_t(x),H_t(y))$ for $(x,y)\in S\disjunt S$. Clearly $\tilde H$ extends continuously to $\tilde H:B\times [0,1]\rightarrow M_{\mathbb B^3}\setminus T^+\mathbb S^2$. Furthermore the bundle $ (\Psi\circ \tilde H_1)^*\pi$ clearly admits a global section $s_1\equiv \wp$. By the covering homotopy theorem, $s_1$ extends to a global section $\tilde s$ of  $(\Psi\circ \tilde H)^*\pi$, and therefore this principal bundle is trivial. This already shows that $E=(\Psi\circ\tilde H_0)^*(\mathcal{L}_{1,2})$ is trivial. Let 
\[
\tilde \tau:(\Psi\circ \tilde H)^*(\mathcal L_{1,2})\rightarrow M_S\times [0,1]\times \mathbb S^1
\]
be the isomorphism corresponding to this global section, i.e. such that $\tilde\tau\circ\tilde s(z,t)=(z,t,1)$. For each $t$, the construction above (cf.\eqref{seccio}) yields  a section $s_t$ of the restriction of $\Psi^*\pi$ to each $\partial M_{S_t}\setminus M_{\partial S_t}$, with $s_1\equiv \wp$, and $s_0=s$. Clearly these fit together to give a global section $\overline s$ of the restriction of $(\Psi\circ \tilde H)^*\pi$ to $\partial M_S\setminus M_C\times[0,1]$. From the construction of $\tilde\tau$ it is clear that the restriction of $\tilde\tau \circ\overline s$ to $\partial M_S\setminus M_C\times \{1\}$ lifts over $q$. Now the covering homotopy theorem implies that $\tilde\tau\circ\overline s$ lifts over all of $\partial M_S\setminus M_C\times[0,1]$. Hence we may take $\tau$ to be the restriction of $\tilde \tau$ to  $(\Psi\circ\tilde H_0)^*(\mathcal{L}_{1,2})=E$.
\end{proof}

While $g$ can not be continuously defined over all $\partial B$, we can consider the continuous extensions of $g$ to $S\disjunt C$ and $C\disjunt S$ respectively. We denote these extensions by $g_1$ and $g_2$ respectively. This way, $\theta(x,y)=g_2(x,y)-g_1(x,y)$ in the notation of Theorem \ref{main}, for every $(x,y)\in C\disjunt C$. Let $T_1\subset B\times \mathbb R$ be the graph of $g$ over $\partial B\setminus C\disjunt C$, completed with the graphs of $g_1$ and $g_2$ over $C\disjunt C$.
Now we sew in a family of vertical intervals over $C\disjunt C$ interpolating these two one-sided limits. To be precise we consider $T_2=C\disjunt C\times[0,1]$ together with the mapping
\begin{eqnarray}\label{t2}
 \sigma\colon T_2&\longrightarrow& C\disjunt C\times \mathbb R\\
(x,y,t)&\longmapsto &\big(x,y,t g_1(x,y)+(1-t)g_2(x,y)\big)\notag
\end{eqnarray}
Note that $\sigma$ is a smooth mapping, possibly non-regular.

In the following we will need to specify some orientations. The manifold $S\disjunt S$, and hence  $M_S$ is canonically oriented by $dS\wedge dS$. This induces an orientation on $\partial M_S$, and hence $T_1$ is naturally oriented. Finally, we choose on $T_2$ the orientation given by $dx\wedge dy\wedge dt$. This way, $T_1$ and $T_2$ induce opposite orientations on the graphs of $g_1$ and $g_2$.

\subsection{Stokes' theorem}
Before applying Stokes' theorem, the non-compacity of $M_S\setminus T^+C$ needs to be settled.  To this end, let us consider the  function $f:M_S\rightarrow [0,\infty]$ which vanishes on $T^+C$, and  assigns to each $z\in M_S\setminus T^+C$ the euclidean distance in $\partial_\infty\mathbb H^3$ between the ideal endpoints of $\Phi(z)$. Here $\mathbb H^3$ denotes again the Poincaré model.  Then
\[
\Delta_\epsilon:=f^{-1}([0,\epsilon))
\]
is a neigborhood of $T^+C$ inside $M_S$, and $M_S\setminus \Delta_\epsilon$ is compact. By Sard's theorem, for almost every $\epsilon$, the level set $\partial \Delta_\epsilon:=f^{-1}(\epsilon)$ is smooth and transverse to $\partial M_S$. Therefore  $M_S\setminus \Delta_\epsilon$ is a compact manifold with corners for almost every $\epsilon>0$. We denote this manifold by $B_\epsilon:=M_S\setminus \Delta_\epsilon$.

Let us consider $T_{1,\epsilon}=T_1\setminus \Delta'_\epsilon$, being $\Delta'_\epsilon=\pi^{-1}(\Delta_\epsilon)$. Here $\pi:B\times\mathbb R\rightarrow B$ is the projection on the first factor. For a generic $\epsilon>0$, Sard's theorem applied to $f\circ \pi$ ensures that $T_{1,\epsilon}$ is a compact manifold with corners. Also $T_{2,\epsilon}=T_2\setminus\sigma^{-1}(\Delta'_\epsilon)$ is a compact manifold with corners for almost every $\epsilon$. 

Since $T_{2,\epsilon}$ can be triangulated, we may think of $(T_{2,\epsilon},\sigma)$ as a (smooth) singular chain. Also $T_{1,\epsilon}$ can be thought of as a singular chain. Hence it makes sense to consider $T_\epsilon:=T_{1,\epsilon}+T_{2,\epsilon}$ as a chain in $\partial B\times\mathbb R\setminus \Delta_\epsilon'$. Its boundary is a singular chain of $\partial\Delta'_\epsilon:=\pi^{-1}\partial\Delta_\epsilon$, namely $\partial T_\epsilon=(T_1\cap \partial\Delta_\epsilon')+\sigma^{-1}(\partial\Delta_\epsilon')$.

In the next subsection, we will construct a chain $R_\epsilon$  in $\partial \Delta'_\epsilon$ such that $\partial R_\epsilon=-\partial T_\epsilon$. This way, $T_\epsilon+R_\epsilon$ is a cycle, and hence gives an element in the homology group $H_3(B_\epsilon\times\mathbb R)$. Since $S$ is contractible, we have the following homotopy equivalences
\[
 B_\epsilon\times\mathbb R\simeq B_\epsilon\simeq B\simeq S\disjunt S\simeq S\times\mathbb S^1\simeq\mathbb S^1.
\]
Therefore $H_3(B_\epsilon\times\mathbb R)=0$, and $T_\epsilon+R_\epsilon$ is a boundary.

By composing with $\pi:B\times \mathbb R\rightarrow B$ we can consider $\pi_*(T_\epsilon+R_\epsilon)$ as a cycle in $(\partial B)\setminus \Delta_\epsilon\cup \partial \Delta_\epsilon=\partial B_\epsilon$. The latter is an oriented compact manifold so $H_3(\partial B_\epsilon,\mathbb Z)\equiv\mathbb Z$, and $[\pi_*(T_\epsilon+R_\epsilon)]$ is given by some integer $n$. For any form $\omega\in \Omega^3(\partial B_\epsilon)$ one has
\[\int_{T_\epsilon+R_\epsilon}\pi_*\omega=
 n\int_{\partial B_\epsilon}\omega.
\]
Note that $\pi$ restricted to the interior of ${T_{1,\epsilon}}$ is a diffeomorphism preserving orientations. Thus, taking $\omega$ supported on the interior of $\pi(T_{1,\epsilon})$ makes clear that $n=1$.

Now, since $H^4(B_\epsilon)=0$, there exists some differential form $\omega\in\Omega^3(B_\epsilon)$ such that $d\omega=\Phi^*d\ell$. Therefore, by Stokes' theorem
\begin{equation}\label{cadena}
\int_{B_\epsilon}\Phi^*d\ell= \int_{B_\epsilon}d\omega=\int_{\partial B_\epsilon} \omega=\int_{T_\epsilon+ R_\epsilon}\pi^*\omega=\frac{1}{2}\int_{T_\epsilon+ R_\epsilon}\pi^*\alpha\wedge \varphi,
\end{equation}
since $2\pi^*\omega-\pi^*\alpha\wedge\varphi$ is closed by \eqref{primitiva}, and $T_\epsilon+ R_\epsilon$ is a boundary. Here we are abusing the notation for simplicity: by $\alpha$ and $\varphi$ we refer to $\Phi^*\alpha$ and $(\Phi'\circ\tau^{-1}\circ\pi )^*\varphi$ respectively. We will go on with this abuse, and hopefully no confusion will arise.

\subsection{Total curvature and ideal defect}
In this section we integrate $\pi^*\alpha\wedge\varphi$ over $T_1$ and $T_2$. We will get respectively the total curvature, and the ideal defect. 
\begin{proposition}\label{tc}
\[
 \lim_{\epsilon\to 0}{\int_{T_1\setminus\Delta'_\epsilon}\pi^*\alpha\wedge\varphi}=2\pi\int_S KdS.
\]
\end{proposition}
\begin{proof}
Recall that 
\[
 T_1=(\mathrm{graph}\ g|_{T^+S})\cup(\mathrm{graph}\ g_1)\cup(\mathrm{graph}\ g_2)
\]

We claim that $\pi^*\alpha\wedge \varphi$ vanishes on the graphs of $g_1$ and $g_2$. Recall these functions are defined over  $S\disjunt C$  and $C\disjunt S$ respectively.
Indeed, let $x$ be a local coordinate on  $C$. Then expression \eqref{alpha1} shows $\alpha\wedge dx=0$. Let now $c(t)$ be the lift in the graph of $g$ of a curve $(y(t),x)\in S\times C$ or $(x,y(t))\in C\times S$ with $x$ fixed. This curve corresponds to a curve $(\ell(t),\wp(t))=\Phi'\circ q(c(t))\in\mathcal L_{1,2}$. In the Poincaré model, the ideal boundaries of $\wp(t)$ are circles in $\partial_\infty\mathbb H^3$ tangent to $C$ at the point $x$. In order to compute $\varphi(c'(t))$ we take an isometry of $\mathbb H^3$ sending the point $x\in C$ to infinity. This way, $\ell(t)$  become vertical lines, and the geodesic planes $\wp(t)$ are transformed into a family of parallel vertical planes. By using the expression \eqref{connexio} of the connection forms, it is clear that $\varphi(c'(t))$ vanishes. This shows that $\varphi$ is a multiple of $\pi^* dx$ (on this region of $T_1$), and the claim follows.

We focus now on the graph over $T^+S$. Given $(x,l)\in T^+S^\circ$,  we take $v_1,v_2,v_3$ an orthonormal basis of $T_x\mathbb H^3$ such that $[v_1]=l$, and $v_3\bot T_xS$.  With such a moving frame, by \eqref{alpha}
\[
 \pi^*\alpha\wedge  \varphi =(\omega_2\wedge\omega_3-\omega_{12}\wedge\omega_{13})\wedge\omega_{23}=-\omega_{12}\wedge\omega_{13}\wedge\omega_{23}=-K(x)\ \omega_{12}\wedge dS.
\]
By Proposition \ref{p7},  this volume form has finite integral on $T^1S$, the euclidean unit tangent of $S$ (in the Poincaré model). Then we may use Lebesgue's dominated convergence theorem to get
\begin{multline*}
 \lim_{\epsilon\to 0}\int_{\mathrm{graph}\ g|_{T^+S}\setminus \Delta'_\epsilon}\pi^*\alpha\wedge \varphi=\lim_{\epsilon\to 0}\int_{T^+S\setminus\Delta_\epsilon}K\ \,\omega_{12}\wedge dS\\=\int_{T^+S}K\ \,\omega_{12}\wedge dS=2\pi\int_{S}KdS,
\end{multline*}
where we used the natural orientation of $T^+S$, which is opposite to the one induced by $M_S$.
\end{proof}

\begin{proposition}\label{id}
\[
  \lim_{\epsilon\to 0}\int_{T_2\setminus\Delta'_\epsilon}\pi^*\alpha\wedge \varphi=2\int_{C\times C} \theta\sin\theta \frac{dxdy}{\|y-x\|^2}.
\]
\end{proposition}
\begin{proof}
Recall that $T_2$ is mapped to the union of vertical segments in $B\times \mathbb R$ interpolating the one-sided limits of $g$ along $C\disjunt C$ (cf. \eqref{t2}). This segments have length $\theta$, and $\varphi$ restricted to the fibers is precisely the length element. Hence, Fubini's theorem gives
\[
 \lim_{\epsilon\to 0}\int_{T_2\setminus\Delta'_\epsilon}\pi^*\alpha\wedge \varphi=\lim_{\epsilon\to 0}\int_{C\times C\setminus\Delta_\epsilon}\theta\  \alpha=2\lim_{\epsilon\to 0}\int_{C\times C\setminus\Delta_\epsilon}\theta\sin\theta\,\frac{dxdy}{\|y-x\|^2}
\]
where we have used \eqref{alpha2}. The result follows since $\theta=O(\|y-x\|)$, which is easy to prove.
\end{proof}
The proof of Theorem \ref{main} is almost finished. So far we have seen (cf. Propositions \ref{p3} and \ref{convabs}, equation \eqref{cadena}, and Propositions \ref{tc} and \ref{id})
\begin{multline}\label{casi}
\int_{\mathcal L}(\#(\ell\cap S)-\lambda^2(\ell, C))d\ell=2\int_{B}\Phi^*(d\ell)=2\lim_{\epsilon\to 0}\int_{B_\epsilon}\Phi^*(d\ell)\\
=\lim_{\epsilon\to 0}\int_{R_\epsilon}\pi^*\alpha\wedge\varphi+2\pi\int_S KdS+2\int_{C\disjunt C}\theta\sin\theta\frac{dxdy}{\|y-x\|^2}.
\end{multline}
It remains only to check that the contribution of $R_\epsilon$ vanishes as $\epsilon\to 0$. This is done in the next subsection.

\subsection{Asymptotic estimations.}
Next we construct a singular chain $R_\epsilon$ in $\partial \Delta'_\epsilon$ with $\partial R_\epsilon=-\partial T_\epsilon$ as promised. Let $v:\partial  \Delta_\epsilon\rightarrow E$ be the section given by the vertical planes. With the same kind of arguments as in the proof of Proposition \ref{P4} one shows that $\tau\circ v$ lifts over $q$; i.e. there exists $h:\partial \Delta_\epsilon\rightarrow \mathbb R$ such that $q(x,h(x))=\tau\circ v(x)$. Let $R_{0,\epsilon}\subset \partial \Delta'_\epsilon$ be the graph of $h$ over $\partial\Delta_\epsilon$.  In particular, both $T_1\cap \partial \Delta'_{\epsilon}$ and $\partial R_{0,\epsilon}$ project by $\pi$ onto $\partial B\cap \partial\Delta_\epsilon$.  Next we consider the union of vertical segments joining these two graphs. More precisely, we define $R_{1,\epsilon}=(C\disjunt S\cap \partial\Delta_\epsilon)\times [0,1]$, $R_{2,\epsilon}=(S\disjunt C\cap \partial\Delta_\epsilon)\times [0,1]$, $R_{3,\epsilon}=(T^+S\cap \partial\Delta_\epsilon)\times [0,1]$ together with the mappings
\begin{eqnarray*}
\sigma_i:R_{i,\epsilon} &\rightarrow& \partial\Delta_\epsilon'\\
(z,t)&\mapsto&(z,tg_i(z)+(1-t)h(z)),
\end{eqnarray*}
for $i=1,2$. As for $\sigma_3$, we take the same definition with $g$ in the place of $g_i$.
We think of $\{R_{i,\epsilon},i=0,1,2,3\}$ as singular chains in $\partial \Delta'_\epsilon$, and we define $R_\epsilon=\sum_{i=0}^3 R_{i,\epsilon}$. A careful study of the boundaries shows that $\partial R_\epsilon=-\partial T_\epsilon$.

To finish the proof of Theorem \ref{main} we only need to establish the following.

\begin{proposition}\label{last}
\[
 \lim_{\epsilon\to 0}\int_{R_\epsilon}\pi^*\alpha\wedge\varphi=0
\]
\end{proposition}
\begin{proof}Here we assume that $S$ coincides with the cylinder $C\times (0,\infty)\subset \mathbb H^3$ in a neighborhood of infinity. This is no loss of generality by Corollary \ref{apriori}, and equation \eqref{casi}. 

In particular we may assume that $h$ and $g$ coincide over $T^+S$. Hence, the integral over $R_{3,\epsilon}$ vanishes. Next we concentrate on $R_{1,\epsilon}$ (the study of $R_{2,\epsilon}$ being obviously symmetric). Given  $(x,y)\in C\times S\cap \partial \Delta_\epsilon$, let $\{x,z\}$ be the ideal endpoints of $\Phi(x,y)$. The euclidean distance between $x$ and $z$ is constant $\epsilon$. Hence, given $x\in C$ the point $z$ is determined by the angle $\gamma$ between the straight segment $\overline{xz}$ and $T_xC$. This angle $\gamma$ coincides with the length of the fiber interval $\pi^{-1}(x,y)\cap R_{1,\epsilon}$. By Fubini's theorem
\[
 \int_{R_{1,\epsilon}}\pi^*\alpha\wedge\varphi=\int_{\pi(R_{1,\epsilon})}\gamma\cdot \alpha=\int_{\pi(R_{1,\epsilon})}\gamma\cos\gamma\frac{dxd\gamma}{\epsilon}
\]
since  $\alpha=\cos\gamma \epsilon^{-1} dxd\gamma$ (cf. \eqref{alpha1}). The previous integrals vanish when $\epsilon\to 0$ since $\gamma=O(\epsilon)$. Indeed, the chords of length smaller than $\epsilon$ make angles with $C$ of order $O(\epsilon)$.
% The latter follows from the following fact which is easy to prove: {\em the chords with length smaller than $\epsilon$ of a simple closed  $\mathcal C^2$-differentiable curve form angles of order $O(\epsilon)$ with that curve.} 

It remains to estimate the integral over $R_{0,\epsilon}$.
Let $(x',y')\in S\disjunt S$ be a generic point in $M_S\cap \partial \Delta_\epsilon$. Let $x,y\in C$ be the vertical projections of $x',y'$. Let $z,w$ be the ideal endpoints of the geodesic $\ell=
\Phi(x',y')$. 
%Assume $z,x',y',w$ are well ordered along $\ell$. 
We choose Euclidean coordinates on $\mathbb R^2\equiv\partial_\infty\mathbb H^3$ so that $x_2=y_2=0$. We can assume $z_1<x_1<y_1<w_1$. Then
\[
dz_2=(t+\sigma)\frac{dx_2}{\sigma}-t\frac{dy_2}{\sigma} 
\]
\[
 dw_2=(\sigma+t-\epsilon)\frac{dx_2}{\sigma}+(\epsilon-t)\frac{dy_2}{\sigma}.
\]
where $t=x_1-z_1$, $\sigma=y_1-x_1$, and thus $w_1-y_1=\epsilon-t- \sigma$. Recall that $(x',y')$ corresponds (through $\Phi'\circ\tau^{-1}\circ\pi$) to the pair $(\ell,\wp)$ where $\wp$ is the vertical plane containing $\ell$. We take an adapted orthonormal frame $(p;g_1,g_2,g_3)$ such that $p\in\mathbb \ell$ projects vertically onto $\frac12(z+w)\in\partial_\infty\mathbb H^3$ and $g_3\bot\wp$. Then \eqref{connexio} and the equations above yield
\[
\varphi=\langle \nabla g_2,g_3\rangle=\theta_{23}=\frac{1}{\epsilon}(dz_2+dw_2)=\frac{1}{\epsilon\sigma}\left(({2t+2\sigma-{\epsilon}})dx_2+({\epsilon}-2t)dy_2\right).
\]
Since $\pi^*\alpha=d\varphi$ we get
\[
 \pi^*\alpha=\frac{2}{\epsilon\sigma}({dt\wedge dx_2}-{dt\wedge dy_2})-\frac{2t-\epsilon}{\epsilon\sigma^2}d\sigma\wedge dx_2-\frac{\epsilon-2t}{\epsilon\sigma^2}d\sigma\wedge dy_2.
\]
Hence, recalling that $x,y$ are restricted to move along $C$, we get
\[
 \pi^*\alpha\wedge\varphi=\frac{4}{\epsilon^2\sigma}dt\wedge dx_2\wedge dy_2=\frac{4}{\epsilon^2}\sin\beta_x\sin\beta_y\frac{dt\wedge dx\wedge dy}{\|y-x\|},
\]
where $dx,dy$ denote arc-length elements on $C$, and $\beta_x,\beta_y$ are angles between $C$ and the segment $\overline{xy}$.
Therefore
\[
 \int_{R_{0,\epsilon}}\pi^*\alpha\wedge\varphi=\frac{4}{\epsilon^2}\int_{x,y\in C,\|y-x\|\leq\epsilon}\sin\beta_x\sin\beta_y(\epsilon-\|y-x\|)\frac{dx\wedge dy}{\|y-x\|}
\]
which goes to zero when $\epsilon\to 0$, since $\beta_x,\beta_y=O(\|y-x\|)$.

\end{proof}

\bibliographystyle{plain}

\end{document}